\documentclass[12pt]{amsart}
\usepackage[colorlinks=true,pagebackref,hyperindex,citecolor=green,linkcolor=red]{hyperref}
\usepackage{amsmath}
\usepackage{amsfonts}
\usepackage{amssymb}
\usepackage{color}
\usepackage[all]{xy}  
\usepackage{enumerate}
\usepackage{verbatim}
\usepackage[top=1in, bottom=1in, left=1in, right=1in]{geometry}
\usepackage{mathrsfs}

\usepackage{stmaryrd}

\theoremstyle{definition}
\newtheorem{theorem}{Theorem}[section]

\newtheorem{question}[theorem]{Question}
\newtheorem{corollary}[theorem]{Corollary}
\newtheorem{lemma}[theorem]{Lemma}
\newtheorem{proposition}[theorem]{Proposition}

\newtheorem{theoremx}{Theorem}

\theoremstyle{definition}

\newtheorem{example}[theorem]{Example}

\newtheorem{remark}[theorem]{Remark}
\numberwithin{equation}{subsection}

\newcommand{\m}{\mathfrak{m}}

\newcommand{\NN}{\mathbb{N}}
\newcommand{\ZZ}{\mathbb{Z}}
\newcommand{\QQ}{\mathbb{Q}}
\newcommand{\CC}{\mathbb{C}}
\newcommand{\FF}{\mathbb{F}}

\newcommand{\fg}{\operatorname{fg}}
\newcommand{\rra}{\operatorname{rra}}

\newcommand{\syz}{\operatorname{syz}}

\newcommand{\Spec}{\operatorname{Spec}}

\newcommand{\Hom}{\operatorname{Hom}}
\newcommand{\End}{\operatorname{End}}
\newcommand{\Ext}{\operatorname{Ext}}
\newcommand{\Tor}{\operatorname{Tor}}

\newcommand{\ara}{\operatorname{ara}}

\newcommand{\Sym}{\mathrm{Sym}}

\renewcommand{\geq}{\geqslant}
\renewcommand{\leq}{\leqslant}

\renewcommand{\hom}[3]{\operatorname{Hom}_{#1} ( #2, #3 )}
\newcommand{\ModDif}[3]{\mathscr{P}^{#1}_{#2 | #3}}
\newcommand{\wModDif}[3]{\widehat{\mathscr{P}}^{#1}_{#2 | #3}}
\newcommand{\Diff}[3]{{D}^{#1}_{#2 | #3}}
\newcommand{\WDiff}[2]{{D}^{\mathrm{cpt}}_{#1 | #2}}
\newcommand{\matl}[1]{^{\vee_{#1}}}

\newcommand{\Sec}{\mathrm{Sec}}
\newcommand{\Proj}{\mathrm{Proj}}
\newcommand{\gHom}{{^*\mathrm{Hom}}}
\newcommand{\ilim}{\varinjlim}
\newcommand{\plim}{\varprojlim}
\newcommand{\df}[3]{R^{#1} D_{#2|#3}}
\newcommand{\resa}{^l \!}
\newcommand{\resb}{^r \!}

\title{Derived Functors of Differential Operators}

\author{Jack Jeffries}\thanks{The author was supported by NSF Grant DMS~\#1606353.}

\begin{document}

\begin{abstract}
In their work on differential operators in positive characteristic, Smith and Van den Bergh define and study the derived functors of differential operators; these arise naturally as obstructions to differential operators reducing to positive characteristic. In this paper, we provide formulas for the ring of differential operators as well as these derived functors of differential operators. We apply these descriptions to show that differential operators behave well under reduction to positive characteristic under certain hypotheses. We show that these functors also detect a number of interesting properties of singularities.
\end{abstract}

\maketitle

\section{Introduction}

The notion of differential operators on singular varieties, as defined by Grothendieck~\cite[\S16.8]{EGAIV4}, has attracted interest in algebra for a number of reasons. Questions about the ring structure of differential operators, e.g., finite generation and simplicity, have been well-studied \cite{BGG,LS,SmithVDB,Schwarz}. In Commutative Algebra, differential operators on singularities have found a resurgence of interest due to their connections with $F$-singularities, symbolic powers, and noncommutative resolutions, among other topics \cite{Smith-D, SmithVDB,Survey-SP,AMHNB,FMS}. While most commonly one encounters the ring $D_{R|K}$ of $K$-linear differential operators from a ring $R$ (commutative, with $1$) to itself, differential operators from one module $M$ to another $N$ are defined; the collection of these is denoted $D_{R|K}(M,N)$. As with $\mathrm{Hom}_R(M,-)$, the functor $D_{R|K}(M,-)$ is left-exact, and it admits a sequence of right-derived functors, denoted $R^i D_{R|K}(M,-)$. In particular, $R^0 D_{R|K}(R,R)=D_{R|K}$. These functors are first studied by Smith and Van den Bergh \cite{SmithVDB}, and we refer to the functors $R^i D_{R|K}(R,-)$ as the \emph{Smith--Van den Bergh functors}. In this paper, we explore and establish a number of new results about rings of differential operators and these Smith--Van den Bergh functors.

The Smith--Van den Bergh functors are motivated by the study of differential operators via reduction to positive characteristic. If $R$ is a finitely generated algebra over a perfect field $K$ of positive characteristic, one has an isomorphism
\begin{equation}\label{eq:poschar} D_{R|K} = \bigcup_{e\in \NN} \End_{R^{p^e}} (R),\end{equation}
where $R^{p^e}\subseteq R$ is the subring of $p^e$-th powers of elements of $R$ \cite[Lemma~1.4.8]{Yekutieli}. This isomorphism is quite useful in understanding rings of differential operators in positive characteristic, especially in conjunction with information on the module-theoretic structure given by the Frobenius map on $R$ \cite{PaulSmith,SmithVDB,TakaTaka,FMS}. Conversely, via this isomorphism, properties of singularities detected by the Frobenius map can sometimes be detected by the ring of differential operators; e.g., for $R$ and $K$ as above, $R$ is strongly $F$-regular if and only if $R$ is simple as a left $D_{R|K}$-module ($R$ is \emph{$D_{R|K}$-simple}) and $F$-pure \cite[Theorem~2.2]{Smith-D}; see the reference for definitions.

Given the utility of the isomorphism \eqref{eq:poschar}, one may hope to deduce results about rings of differential operators in characteristic zero from their analogues in positive characteristic. If $R$ is a finitely generated $\QQ$-algebra, we can find a $\ZZ$-algebra $S$ such that $S \otimes_{\ZZ} \QQ \cong R$. To deduce properties of $D_{R|\QQ}$ by reduction to positive characteristic, one needs to relate $D_{R|\QQ}$ to $D_{S/pS |\FF_p}$. In general, for a flat $\ZZ$-algebra $S$ and a field $K$ there is a natural injective map
\begin{equation}\label{eq:reduction}
D_{S | \ZZ} \otimes_{\ZZ} K \longrightarrow D_{S \otimes K | K}
\end{equation}
that is an isomorphism when $K$ has characteristic zero. Thus, a natural question is when the map \eqref{eq:reduction} is an isomorphism for infinitely many, all but finitely many, or all prime fields $\FF_p=K$. Positive answers not only allow for reduction to positive characteristic techniques, but in light of \eqref{eq:poschar}, give a sort of uniformity of the Frobenius map under arithmetic deformations. Smith and Van den Bergh show that the cokernel of \eqref{eq:reduction} when $K=\FF_p$ is the $p$-torsion part of $R^1 D_{R | \ZZ}(R,R)$; see Proposition~\ref{SVDB-reduction} below. With the aim of extending their result \cite[Theorem~1.3]{SmithVDB} on simplicity of rings of differential operators in positive characteristic, Smith and Van den Bergh ask in \cite[Question~5.1.2]{SmithVDB} if the $p$-torsion part of $R^1 D(R,R)$ vanishes when $R \otimes_\ZZ \QQ$ is a direct summand (as a $R \otimes_\ZZ \QQ$-module) of a polynomial ring over $\QQ$.

In this paper, we give two descriptions of the differential operators $D_{R | A}(R,\omega_R)$ and the Smith--Van den Bergh functors $R^i D_{R | A}(R,\omega_R)$ for certain Cohen-Macaulay $A$-algebras with canonical module $\omega_R$; in particular, when $R$ is Gorenstein, these formulas describe the ring of differential operators on $R$ as well as the aforementioned $R^1 D_{R | A}(R,R)$. Our formulas, which we state in the Gorenstein graded $K$-algebra case for convenience, are in terms of local cohomology modules.
\begin{theoremx}\label{THM-A}
	Let $K$ be a field, and $R$ be a Gorenstein graded $K$-algebra of Krull dimension~$d$.
	\begin{enumerate}
		\item\label{A-1} (Theorem~\ref{functors-field}) For all $i \geq 0$, $R^i D_{R|K} (R,R) \cong H^{d+i}_{\Delta}(R\otimes_K R)$, where $\Delta$ is the ideal of the diagonal.
		\item\label{A-2} (Theorem~\ref{second-iso}) If $R$ has an isolated singularity, $d\geq 3$, and $1\leq i \leq d-2$, then $R^i D_{R|K} (R,R) \cong H^{i+1}_{R_+}(D_{R|K})$,
		where $R_+$ is the maximal homogenous ideal.\qedhere
	\end{enumerate}
\end{theoremx}\noindent
In particular, when $i=0$, Theorem~\ref{THM-A}(\ref{A-1}) yields the isomorphism 
\[D_{R|K} \cong H^{d}_{\Delta}(R\otimes_K R),\]
 which, for smooth varieties, is known as the Grothendieck--Sato formula. This description of $D_{R|K}$ was previously obtained in \cite{BZN} for Cohen-Macaulay algebras over a field satisfying a certain cohomological assumption; our characterization does not require this cohomological hypothesis.

In the case $i=1$, an immediate consequence of Theorem~\ref{THM-A}(\ref{A-1}) is that differential operators behave well under reduction to positive characteristic if and only if there is no $p$-torsion in a particular local cohomology module of $R\otimes_K R$. The study of $p$-torsion in local cohomology is of independent interest for its connections to Huneke's and Lyubeznik's questions on finiteness of associated primes of local cohomology \cite{Singh-LC,Huneke?,Lyubeznik?}. We use this connection to obtain the following positive answer to \cite[Question~5.1.2]{SmithVDB} under a stronger hyopthesis:
\begin{theoremx}[see Theorem~\ref{splitting-reduction}] Let $R$ be a Gorenstein graded $\ZZ$-algebra. If $R$ is a direct summand of a polynomial ring as an $R$-module, then the map $D_{R|\ZZ} \otimes \ZZ /p\ZZ \rightarrow D_{R/pR | \FF_p}$ is an isomorphism for all but at most finitely many primes $p$. 
\end{theoremx}
\noindent In particular, this applies to Gorenstein toric rings, as well as invariant rings of finite groups (after possibly replacing $\ZZ$ by $\ZZ[1/n]$ for some $n$). Additionally, we apply our description to give an easy conceptual example of a local cohomology module with $p$-torsion for all primes $p$; see Example~\ref{assoc-primes}. We also show that there is a precise obstruction to the property that a $K$-algebra $R$ is simple as a left $D_{R|K}$-module that lives in $R^1 D_{R|K}$.

We also use Theorem~\ref{THM-A} to show that the vanishing or nonvanishing of $R^i D_{R|K}(R,R)$ for $i> 1$ has interesting connections to aspects of the singularity of $R$. If $R$ is an isolated Gorenstein singularity, the (non)vanishing of these modules characterizes the depth of $D_{R|K}$ as an $R$-module; see Theorem~\ref{depth-cor}. Nonvanishing of these modules yields a lower bound for the dimension of secant varieties; see Proposition~\ref{radrank}. As an application of these results, we provide a result on injective maps from affine toric varieties into affine spaces; see Example~\ref{example-simplicial}.

\section{Differential operators and their derived functors}

Throughout, $A$ and $R$ denote commutative Noetherian rings with $1$.

\subsection*{Local cohomology}

For an ideal $I$ in a ring $R$, the \emph{$I$-torsion functor} \[\Gamma_I:R-\mathrm{mod}\rightarrow R-\mathrm{mod}\] is defined as ${\Gamma_I(M)=\bigcup_{n\in \NN}  \mathrm{Ann}_{M} (I^n)}$. The \emph{local cohomology functors} with support in $I$ are the right derived functors of the $I$-torsion functor: $H^i_I(-):= R^i \Gamma_I (-)$. A useful equivalent description is $H^i_I(M):=\ilim \Ext^i_R(R/I^n,M)$. Additionally, if $I=(f_1,\dots,f_t)$ the $i$-th local cohomology can be computed as the $i$-th cohomology of a complex of the form
\[ 0 \to R \to \bigoplus_{1\leq i \leq n} R_{f_i} \to\bigoplus_{1 \leq i<j \leq n} R_{f_i f_j} \to \cdots \to R_{f_1 \cdots f_t} \to 0,\]
where the maps are the sums of the natural localization maps, up to sign.

We refer the reader to \cite[\S3]{BH} and \cite{24} for the equivalences above, and other basic properties of local cohomology. 

\subsection*{Differential operators}

Let $R$ be a ring, and $A$ be a subring. If $M$ and $N$ are two $R$-modules, the \emph{($A$-linear) differential operators from $M$ to $N$ of order at most $i$}, denoted $\Diff{i}{R}{A}(M,N)$, are defined inductively as follows:
\begin{itemize}
	\item $\Diff{0}{R}{A}(M,N)=\hom{R}{M}{N}$,
	\item $\Diff{i}{R}{A}(M,N)=\{ \delta \in \hom{A}{M}{N} \ | \  \forall f\in R, \ \delta \circ f_M - f_N \circ \delta \in \Diff{i-1}{R}{A}(M,N) \}$,
\end{itemize}
where $f_M$ and $f_N$ denote multiplication by $f$ on $M$ or $N$ respectively. The \emph{differential operators from $M$ to $N$}   are 
\[ \Diff{}{R}{A}(M,N)=\bigcup\limits_{i=0}^{\infty} \Diff{i}{R}{A}(M,N).\]
Throughout we denote $\Diff{}{R}{A}(R,M)$ by $\Diff{}{R}{A}(M)$ and $\Diff{}{R}{A}(R,R)$ by $\Diff{}{R}{A}$, and similarly for differential operators up to a certain order.

Let $L$, $M$, and $N$ be three $R$ modules. It is easily verified that if $\alpha \in \Diff{i}{R}{A}(L,M)$ and $\beta \in \Diff{j}{R}{A}(M,N)$, then $\beta \circ \alpha \in  \Diff{i+j}{R}{A}(L,N)$. Consequently, $\Diff{}{R}{A}$ is a ring under composition, and $\Diff{}{R}{A}(M)$ is a right $\Diff{}{R}{A}$-module for any $R$-module $M$. Of course, $R$ is a left $\Diff{}{R}{A}$-module as well. We say that $R$ is \emph{$\Diff{}{R}{A}$-simple} if $R$ is a simple left $\Diff{}{R}{A}$-module.

\subsection*{Modules of principal parts}

We define the \textit{enveloping algebra} of $R$ over $A$ to be
\[\ModDif{}{R}{A}:=R\otimes_A R,\] and the
\textit{module of $n$-principal parts} of $R$ over $A$ to be
 \[\ModDif{n}{R}{A}:=\ModDif{}{R}{A}/\Delta^{n+1}_{R|A},\] 
where $\Delta_{R|A}$ is the kernel of the multiplication map $\ModDif{}{R}{A}\stackrel{\mu}{\longrightarrow} R$. We also set
\[\ModDif{}{R}{A}(M):=R\otimes_A M, \quad \text{and} \quad \ModDif{n}{R}{A}(M):=\ModDif{}{R}{A}(M) \otimes_{\ModDif{}{R}{A}} \ModDif{n}{R}{A}. \]
By \cite[16.7.4]{EGAIV4}, if $R$ is essentially of finite type over $A$ and $M$ is a finitely generated $R$-module, then $\ModDif{n}{R}{A}(M)$ is a finitely generated $R$-module for any $n$.

There is a natural isomorphism of Hom-tensor adjunction
\[ \Hom_R(\ModDif{}{R}{A}(M), N) \cong \Hom_A(M,N), \]
given by $\phi \mapsto \phi \circ d$, where $d(x)=1\otimes x$, and the homomorphisms on the left-hand-side are those linear with respect to the left $R$-action on $\ModDif{}{R}{A}(M)$.
Under this isomorphism, the $\ModDif{}{R}{A}$-module structure on $\Hom_R(\ModDif{}{R}{A}(M), N)$ given by its action on the source corresponds to the $\ModDif{}{R}{A}$-module structure on $\Hom_A(M,N)$ where the left factor acts on $N$ and the right factor acts on $M$. The adjunction map induces isomorphisms
\[ \Hom_R(\ModDif{n}{R}{A}(M), N) \cong \Diff{n}{R}{A}(M,N) \]
and
\[\ilim\hom{R}{\ModDif{n}{R}{A}(M)}{N} \cong \Diff{}{R}{A}(M,N).\]
These are isomorphisms of $\ModDif{}{R}{A}$-modules, where the structures are as described above. 

\begin{remark} We will often wish to consider $\ModDif{}{R}{A}$-modules as $R$-modules; in general this can be done in different ways.
	For a $\ModDif{}{R}{A}$-module $\mathcal{M}$, we will write $\resa \mathcal{M}$, respectively $\resb \mathcal{M}$, for the $R$-module structure on $\mathcal{M}$ given by restriction of scalars to the left, respectively right, $R$-factor of $\ModDif{}{R}{A}$. In particular, $\resa\Diff{}{R}{A}(M,N)$ is the structure induced by the action on $N$, and $\resb\Diff{}{R}{A}(M,N)$ is the structure induced by the action on $M$.
\end{remark}

\subsection*{The Smith--Van den Bergh functors}

The functor $M\mapsto \Diff{}{R}{A}(M)$ is a left-exact functor from the category of $R$-modules to $\ModDif{}{R}{A}$-modules. Following \cite[\S2.2]{SmithVDB}, we denote its right derived functors by $\df{i}{R}{A}(M)$. There is a natural identification \[{\df{i}{R}{A}(M)\cong\ilim\Ext^i_R(\ModDif{n}{R}{A},M)}.\]
Under mild hypotheses, there is an alternative description.

\begin{proposition}[Smith--Van den Bergh {\cite[Proposition~2.2.2]{SmithVDB}}]\label{SVDB-description} Suppose that $R$ is projective over $A$ and $\ModDif{}{R}{A}$ is Noetherian. Then $\df{i}{R}{A}({M})\cong H^i_{\Delta_{R|A}}(\hom{A}{R}{M})$ as $\ModDif{}{R}{A}$-modules. Consequently, if $\ModDif{}{R}{A}$ has Krull dimension $D<\infty$, then $\df{i}{R}{A}(-)\equiv 0$ for $i>D$.
\end{proposition}

We note the following consequence.

\begin{lemma}\label{dcd} Suppose that $R$ is projective over $A$ and $\ModDif{}{R}{A}$ is Noetherian and finite-dimensional. Let $s\in \NN$.
	If $\df{i}{R}{A}(R)=0$ for all $i>s$, then $\df{i}{R}{A}(-)\equiv 0$ for all $i>s$. In particular, if $\df{i}{R}{A}(R)=0$ for all $i>0$, then the functor $D_{R|A}(-): R\mathrm{-mod} \rightarrow \ModDif{}{R}{A}\mathrm{-mod}$ is exact.
\end{lemma}
\begin{proof} Since the functors $\df{i}{R}{A}$ commute with direct sums,  $\df{i}{R}{A}(F)=0$ for any free module $F$. Given an arbitrary module $M$, take a short exact sequence \[0 \rightarrow \syz(M) \rightarrow F \rightarrow M \rightarrow 0.\] The long exact sequence in the functors $\df{\bullet}{R}{A}$ gives isomorphisms 
	\[\df{i}{R}{A}(M)\cong \df{i+1}{R}{A}(\syz(M)).\]
	Applying this repeatedly gives isomorphisms \[\df{i}{R}{A}(M)\cong \df{i+t}{R}{A}(\syz^t(M)).\]
	But, by Proposition~\ref{SVDB-description}, for $i+t>\mathrm{dim}(\ModDif{}{R}{A})$, we have $\df{i+t}{R}{A}(-)\equiv 0$, so ${\df{i}{R}{A}(-)\equiv 0}$.
	
	For the final statement, we note that the left-exact functor $D_{R|A}(-)$ is exact if and only if all of its right-derived functors vanish.
\end{proof}

\begin{remark}\label{CD}
	The previous lemma is the analogue of a well-known property of local cohomology: for any ideal $I$ in a ring $R$, $H^{i}_I(R)=0$ for $i>s$ implies $H^i_I(-)\equiv 0$ for $i>s$ \cite[Theorem~9.2]{24}.
\end{remark}

We will refer to the functors $\df{i}{R}{A}(-)$ as the \emph{Smith--Van den Bergh functors}.

\section{Duality and $D$-simplicity}

Recall that if $(R,\m,K)$ is a local ring, the \emph{Matlis duality} functor from $R$-mod to $R$-mod is given by $(-)\matl{}=\Hom_R(-,E_R(K))$, where $E_R(K)$ is the injective hull of the residue field. Similarly, if $R$ is $\NN$-graded, with $R_0$ a field $K$, the \emph{graded Matlis duality} functor is given by $(-)\matl{}=\gHom_K(-,K)$, which agrees with the usual Matlis duality functor on the category of graded $R$-modules; see \cite[Proposition~3.6.16]{BH}. Here, we use the notation $\gHom$ to denote the module of graded homomorphisms. We note that if $M$ is a finitely generated module and $N$ is an arbitrary graded module, $\gHom(M,N)$ agrees with $\Hom(M,N)$ after forgetting the grading, and likewise for the derived functors $^*\Ext(M,N)$, see \cite[Section~1.5]{BH}.

If $R$ is a local ring with a coefficient field, i.e., a subfield isomorphic to $K$ (which we will identify with $K$), we have the following useful alternative description: let $(-)\matl{K}=\ilim \Hom_K(M/\m^n M, K)$. Then the functor $(-)\matl{K}$ agrees with $\matl{}$ on the category of finitely generated $R$-modules \cite[Remark~3.2]{SmithBase}. Similarly, if $R$ is $\NN$-graded, with $R_0$ a field $K$, we set $(-)\matl{K}=\gHom_K(-,K)$, and in this setting the functor $(-)\matl{K}$ agrees with $\matl{}$ on the category of finitely generated graded $R$-modules.

We will use a version of a result of Switala \cite{Swi} in the complete local case and Switala--Zhang \cite{SwiZha} in the graded case that asserts that the Matlis dual of a left $\Diff{}{R}{K}$-module is a right $\Diff{}{R}{K}$-module, and vice versa. We provide here an alternative interpretation of this result, with an auxiliary observation that we employ in the sequel.

Given an $R$-module $M$, let $\fg(M)$ denote the directed system (with respect to containment) of finitely generated submodules of $M$. The \emph{compactly supported differential operators} from $M$ to $N$ are 
\[ \WDiff{R}{A}(M,N):= \plim\limits_{M_\lambda \in \fg(M)} \Diff{}{R}{A}(M_\lambda,N). \]
If $M$ is finitely generated, then $\WDiff{R}{A}(M,N)=\Diff{}{R}{A}(M,N)$ for any module $N$.

The first part of the following proposition is well-known to experts, but we include it here for completeness.

\begin{proposition}\label{Switala-duality}
	Let $(R,\m,K)$ be a graded or local ring essentially of finite type over a coefficient field $K$.
	\begin{enumerate}
		\item If $M$ is a finitely generated $R$-module, then $\resb\Diff{}{R}{K}(M,K)\cong M\matl{}$ as $R$-modules.
		\item If $M$ is an arbitrary $R$-module, then $\resb\WDiff{R}{K}(M,K)\cong M\matl{}$ as $R$-modules.
		\item (Switala, Switala--Zhang) If $M$ is an $R$-module with a compatible left (right) $\Diff{}{R}{K}$-module structure, then the $R$-module $M\matl{}$ admits a natural right (respectively, left) $\Diff{}{R}{K}$-module structure.\qedhere
	\end{enumerate}
\end{proposition}
\begin{proof}
	\begin{enumerate}
		\item We compute as $R \otimes_K R$-modules:
		\begin{align*}	\Diff{}{R}{K}(M,K)
		&\cong\ilim \hom{R}{\ModDif{n}{R}{K}(M)}{K} \\
		&\cong \ilim \mathrm{Hom}_{R}\Big(\frac{R\otimes_K M} {\Delta^{n+1}_{R|K}(R\otimes_K M)},K \Big)  \\
		&\cong \ilim \mathrm{Hom}_{R}\Big(\frac{R\otimes_K M} {(\m \otimes 1 + \Delta^{n+1}_{R|K})(R\otimes_K M)},K \Big) \\
		&\cong \ilim \hom{R}{K\otimes_K M/\m^{n+1}M}{K},
		\end{align*}
		where the third equality follows from the fact that any $R$-linear map from a module $N$ to $K$ factors uniquely through $N/\m N$, and the fourth from the equality 
		\[(\m \otimes 1 + \Delta^{n+1}_{R|K}) = (\m \otimes 1 + 1\otimes \m^{n+1})\] of ideals in $R\otimes_K R$. Here, we have kept the notation ${K\otimes_K M/\m^{n+1}M}$ to indicate the $R\otimes_K R$-module structure.
		Thus, 
		\begin{align*}\resb\Diff{}{R}{K}(M,K) &= \resb \ilim \hom{R}{K\otimes_K M/\m^{n+1}M}{K} \\ &= \ilim \hom{K}{M/\m^{n+1}M}{K}\cong M\matl{}.
		\end{align*}
		
		\item We compute
		\begin{align*} \resb\WDiff{R}{K}(M,K) 
		&= \plim\limits_{M_\lambda \in \fg(M)} {\resb\Diff{}{R}{K}(M_\lambda,K)} \\
		&\cong \plim\limits_{M_\lambda \in \fg(M)} M_\lambda\matl{} \\
		&\cong \Big(\ilim\limits_{M_\lambda \in \fg(M)} M_\lambda\Big)\matl{} \cong M\matl{}.
		\end{align*}
		\item First, we note that if $M$ is a left (respectively, right) $\Diff{}{R}{K}$-module, then the action of $\Diff{}{R}{K}$ is given by differential operators from $M$ to $M$; indeed, this is immediate from the definition. Now, every differential operator $\delta$ on $M$ takes finitely generated submodules to finitely generated submodules. Indeed, $\delta$ factors as $d:M\rightarrow \ModDif{n}{R}{K}(M)$ followed by an $R$-linear map $\phi: \ModDif{n}{R}{K}(M)\rightarrow M$; the restriction of such $\delta$ to a finitely generated $M_{\lambda}$ factors through an $R$-linear map $\phi':\ModDif{n}{R}{K}(M_\lambda)\rightarrow M$, and since $\ModDif{n}{R}{K}(M_\lambda)$ is finitely generated, the image is contained in a finitely generated submodule of $M$. Consequently, $\Diff{}{R}{K}(M,M)$ acts on $\WDiff{R}{K}(M,K)$ by precomposition. Hence, we obtain an action of ${(\Diff{}{R}{K})}^{\mathrm{op}}$ on $\WDiff{R}{K}(M,K)$.\qedhere
	\end{enumerate}
\end{proof}

We apply the previous proposition to show that \emph{$D_{R|K}$-simplicity}, the property that $R$ is a simple left $D_{R|K}$-module, admits a straightforward characterization in terms of the Smith--Van den Bergh functor $\df{1}{R}{K}(-)$.

\begin{proposition}\label{D-simple-prop}
	Let $K$ be a field, and $R$ be a Noetherian $\NN$-graded $K$-algebra, with $R_0=K$, and maxixal homogeneous ideal $R_+$. Let $\varpi:\Diff{}{R}{K}(R)\rightarrow \Diff{}{R}{K}(K)$ be the map obtained by composition with the $R$-linear surjection $\pi:R\rightarrow K$. Then $R$ is a simple left $D_{R|K}$-module if and only $\varpi$ is surjective. Equivalently, $R$ is $D_{R|K}$-simple if and only if the connecting homomorphism $\Diff{}{R}{K}(K)\rightarrow \df{1}{R}{K}(R_+)$ is the zero map.
\end{proposition}

\begin{proof}
	By Proposition~\ref{Switala-duality}~(1), $\resb\Diff{}{R}{K}(K)$ is isomorphic to $E_R(K)$. Clearly, $\eta=\pi \circ d$ in $\Diff{}{R}{K}(K)$ generates the socle of $R\matl{}$. Now, $\varpi(\delta)=\eta \circ \delta$; that is, $\varpi$ agrees with the natural right $D_{R|K}$-action on $E_{R}(K)$ applied to a socle generator of $E_{R}(K)$.
	
	We claim that $E_{R}(K)$ is a simple right $D_{R|K}$-module if and only if $\varpi$ is surjective. Indeed, the image of $\varpi$ is the $D_{R|K}$-submodule of $E_{R}(K)$ generated by the socle, so if  $E_{R}(K)$ is simple, then $\varpi$ is surjective. On the other hand, any nonzero $D_{R|K}$-submodule of $E_{R}(K)$ contains the socle of $E_R(K)$ and hence must contain the image of $\varpi$. Thus, if $\varpi$ is surjective, $E_R(K)$ is simple.
	
	To conclude the proof, we note by Proposition~\ref{Switala-duality}~(3) and faithfulness of Matlis duality, $R$ has a nontrivial left $D_{R|K}$-submodule if and only if $E_{R}(K)$ admits a nontrivial right $D_{R|K}$-submodule. Thus, $E_R(K)$ is simple if and only if $R$ is as well.
\end{proof}

\section{Two descriptions of $\df{i}{R}{A}(\omega_R)$}

We give our first description of the Smith--Van den Bergh functors for graded algebras. We refer to \cite[\S 1.5~and~\S 3.6]{BH} for background on the graded local duality used below.

\begin{theorem}\label{functors-field}
	Let $A$ be a field or complete local ring, and $R$ be a Cohen-Macaulay Noetherian $\NN$-graded $A$-algebra, with $R_0=A$, and maximal homogeneous ideal $R_+$. Let $d=\mathrm{height}(R_+)$. Then $\df{i}{R}{A}(\omega_R)\cong H^{d+i}_{\Delta_{R|A}}(\omega_{\ModDif{}{R}{A}})$ as $\ModDif{}{R}{A}$-modules.
\end{theorem}
\begin{proof}
	First, we note that $R$ and $P:=\ModDif{}{R}{A}$ are quotients of polynomial rings over $A$, so each has a canonical module. The hypotheses on $A$ imply that $R$ and $P$ are each *local *complete rings, in the terminology of \cite[\S 1.5]{BH}, with *maximal ideals $R_+$ and $P_+$, respectively. The *dimensions of $R$ and $P$ are $d$ and $2d$, respectively.

	Since each of the graded $R$-modules $\ModDif{n}{R}{A}$ is finitely generated, we may apply graded local duality over $R$ to obtain $\Ext^i_R(\ModDif{n}{R}{A},\omega_R)\cong (H^{d-i}_{R_+}(\ModDif{n}{R}{A}))\matl{}.$
	
	We note that $P$ is also a Cohen-Macaulay ring. We can then apply graded local duality over $P$, yielding isomorphisms $ \Ext^{d+i}_{P}(\ModDif{n}{R}{A},\omega_{P})\cong (H^{d-i}_{P_+}(\ModDif{n}{R}{A}))\matl{}$. Since $\ModDif{n}{R}{A}$ is finitely generated as an $R$-module and as a $P$-module, $H^{d-i}_{P_+}(\ModDif{n}{R}{A})$ and $H^{d-i}_{R_+}(\ModDif{n}{R}{A})$ both agree with the local cohomology of $\ModDif{n}{R}{A}$ with support in its own homogeneous maximal ideal, hence are naturally isomorphic.
	Thus, there are isomorphisms $\Ext^i_R(\ModDif{n}{R}{A},\omega_R)\cong \Ext^{d+i}_{P}(\ModDif{n}{R}{A},\omega_{P})$ that extend to isomorphisms of directed systems as $n$ varies by functoriality of local duality. Each of these isomorphisms is $P$-linear. The statement then follows from passing to direct limits of the directed systems.
\end{proof}

This description yields an interesting symmetry in the case when $R$ is Gorenstein.

\begin{corollary}\label{gorollary}
	Let $R$ and $A$ be as above, and assume that $R$ is Gorenstein. Then there is an isomorphism $\resa \df{i}{R}{A}(R) \cong {\resb\df{i}{R}{A}(R)}$. In particular, the source and target $R$-module structures on the ring of differential operators $\Diff{}{R}{A}$ agree.
\end{corollary}
\begin{proof}
When $R$ is Gorenstein, $\ModDif{}{R}{A}$ is as well, so $\omega_{\ModDif{}{R}{A}}\cong \ModDif{}{R}{A}$. Then, the module ${\df{i}{R}{A}(R)\cong H^{d+i}_{\Delta_{R|A}}(R\otimes_A R)}$ admits the same $R$-module structure by restriction of scalars to either $R$-factor.
\end{proof}

\begin{example}
	The conclusion of the previous corollary may fail if $R$ is not Gorenstein. Let $K$ be a field, and $R=K[x,y]/(x,y)^2$. Then $D_{R|K}=\Hom_K(R,R)$, since $\Delta_{R|K}^3=0$ in $\ModDif{}{R}{K}$. Thus, $\resa{D_{R|K}}\cong \Hom_K(K,R)^{\oplus 3}\cong R^{\oplus 3}$, whereas $\resb{D_{R|K}}\cong \Hom_K(R,K)^{\oplus 3}\cong (R\matl{})^{\oplus 3}$. Since $R$ is not Gorenstein, these are not isomorphic. The same argument works more generally for any graded artinian $K$-algebra that is not Gorenstein.
\end{example}

\begin{remark}\label{BZNrmk} Our description of $\df{i}{R}{A}(R)$ also sheds some light on the hypotheses of \cite{BZN}. Let $K$ be a field. In {ibid.}, a Cohen-Macaulay $K$-algebra of Krull dimension $d$ determines a \emph{good CM variety} if ${H^{d+1}_{\Delta_{R|K}}(R\otimes_K \omega_R)}=0$, and a \emph{very good CM variety} if $H^{i}_{\Delta_{R|K}}(M)=0$ for all $\ModDif{}{R}{K}$-modules $M$ and all $i>d$. If $R$ is Gorenstein and graded, it follows from Theorem~\ref{functors-field} and Lemma~\ref{dcd} that $R$ is a good CM variety if and only if $\df{1}{R}{K}(R)=0$, and $R$ is a very good CM variety if and only if $\Diff{}{R}{K}(-)$ is an exact functor. We note also that in {ibid.}, the Grothendieck--Sato formula is obtained for good CM varieties without the Gorenstein assumption.
\end{remark}

Using a version of a duality of Horrocks~\cite{Horrocks} due to Dao and Monta\~no \cite{DM}, we can give a different description of the derived functors of the ring of differential operators for isolated singularities.

\begin{theorem}\label{second-iso} Let $K$ be a field, and $R$ be a graded $K$-algebra with an isolated singularity. Suppose that $d=\dim(R)\geqslant 3$. Then,  $\df{i}{R}{K}(\omega_R)\cong H^{i+1}_{R_+}(\Diff{}{R}{K}(\omega_R))$ as $\ModDif{}{R}{K}$-modules for $1\leqslant i \leqslant d-2$.
\end{theorem}
\begin{proof} By \cite[Proposition~2.2]{DM}, if $R$ is Cohen-Macaulay of Krull dimension at least 3, $M$ is a finitely generated module that is free on the punctured spectrum, and $1\leq i \leq d-2$, then there are isomorphisms $\Ext^i_R(M,\omega_R)\cong H^{i+1}_{R_+}(\Hom_R(M,\omega_R))$. Since the formation of $\ModDif{n}{R}{K}$ commutes with localization, $\ModDif{n}{R}{K}$ is free on the punctured spectrum. Thus, there are isomorphisms
	\[\Ext^i_R(\ModDif{n}{R}{K},\omega_R)\cong H^{i+1}_{R_+}\big(\Hom_R(\ModDif{n}{R}{K},\omega_R)\big)\cong H^{i+1}_{R_+}(\Diff{n}{R}{K}(\omega_R)) \] for each $n$. The statement follows from passing to the direct limit.
\end{proof}

\begin{remark} While Theorems~\ref{functors-field}~and~\ref{second-iso} are stated in the graded case, the analogous statements hold in the complete case. If $R$ is complete with coefficient field $K$, one replaces $\ModDif{}{R}{K}$ with its $\m$-adic completion $\wModDif{}{R}{K}\cong R\widehat{\otimes}_K R$, and $\ModDif{n}{R}{K}$ with $\wModDif{n}{R}{K}$, and the same proofs hold mutatis mutandis.
\end{remark}

\section{Vanishing criteria}

We recall that if $M$ is an $R$-module, and $I$ an ideal of $R$, there is an equality
\[ \inf\{n\in \NN \ | \   \Ext^n_R(R/I,M)\neq 0    \} = \inf\{n\in \NN \ | \    H^n(I;M)\neq 0      \} = \inf\{n\in \NN \ | \    H^n_I(M)\neq 0       \},  \]
where $H^n(I;M)$ denotes the Koszul cohomology \cite[Theorem~9.1]{24}. We call this common value the \emph{depth} of $I$ on $M$; if $R$ is local (or $\NN$-graded), the \emph{depth} of $M$ is the depth of the maximal (homogenous) ideal on $M$. If $M$ is finitely generated, this agrees with the usual notion of depth characterized by regular sequences.

As an easy corollary of Theorem~\ref{second-iso}, we see that the Smith--Van den Bergh functors determine the depth of the ring of differential operators for isolated Gorenstein singularities.

\begin{corollary}\label{depth-cor}
	Let $K$ be a field, and $R$ be a Gorenstein graded $K$-algebra that is a domain with an isolated singularity. Suppose that $d=\dim(R)\geqslant 3$. Then, 
	\[\mathrm{depth}_R(D_{R|K}) = \begin{cases} 1+ \min\{ i>0 \ | \ \df{i}{R}{K}(R)\neq 0 \} \quad &\text{if this number is $\leq d$}, \\
	d \quad &\text{otherwise}.\end{cases} \]\qedhere
\end{corollary}
\begin{proof} The hypotheses imply that $R$ is normal. Since $R$ has depth at least two, the modules of differential operators $D^n_{R|K}\cong \Hom_R(\ModDif{n}{R}{K},R)$ of bounded order do as well \cite[{Tag 0AV5}]{stacks}. Thus, $H^0_{R_+}(D^n_{R|K})=H^1_{R_+}(D^n_{R|K})=0$ for each $n$, where $R_+$ is the maximal homogeneous ideal; one has $H^0_{R_+}(D_{R|K})=H^1_{R_+}(D_{R|K})=0$ by taking direct limits. The formula then follows from Theorem~\ref{second-iso} and the characterization of depth in terms of local cohomology.
\end{proof}

We also apply Theorem~\ref{functors-field} to give vanishing conditions on $\df{i}{R}{K}(\omega_R)$ in terms of meaningful geometric information.
Recall that a map of $K$-algebras $R\rightarrow S$ is \emph{radicial} if, for some (equivalently, every) algebraically closed field $L\supseteq K$, the map  $\Spec(S \otimes_K L) \rightarrow  \Spec(R \otimes_K L)$ is injective \cite[Definition~3.5.4]{EGAI}; in particular, if $K$ is algebraically closed, this is equivalent to $\Spec(S) \rightarrow  \Spec(R)$ being injective. We note that a radicial map of graded rings or local rings is necessarily module-finite by Nakayama's lemma, and hence is bijective on spectra.

We define the \emph{radicial rank} of a $K$-algebra $R$ to be
\[\rra_K(R)=\min\{ n \ | \ \exists f_1,\dots,f_n\in R \ \text{such that} \ K[f_1,\dots,f_n]\rightarrow R \ \text{is radicial} \}. \]
Since the arithmetic rank of an ideal $I\subseteq R$ can be characterized as 
\[ \ara(I)=\min\{ n \ | \ \exists f_1,\dots,f_n\in R \ \text{such that} \ R/(f_1,\dots,f_n) \rightarrow R/I \ \text{is bijective on Spec} \}, \]
we think of radicial rank as a subalgebra analogue of arithmetic rank. Geometrically, for affine varieties $X$ over algebraically closed fields, the radicial rank is the smallest dimension of an affine space to which $X$ maps injectively (as set-theoretic maps of $\Spec$).

If $R$ is a standard graded ring, so that $R=S/I$ for some standard graded polynomial ring, and $I$ a homogeneous ideal, then the quotient map $S\to R$ induces an embedding of $X=\Proj(R)$ into a projective space. The radicial rank provides a lower bound for the dimension of the secant variety of this embedding: $\rra_K(R)\leq \dim(\Sec(X))+1$. This is well-known; see, e.g., \cite[Corollary~4.3]{DufJef2}. We note that this inequality can be sharp; see ibid.~for multiple examples. We also note that the radicial rank is independent of the presentation of $R$, and does not require any a priori graded structure.

The following proposition is closely related to \cite[\S3]{DufJef} and \cite[\S4]{DufJef2}.
\begin{proposition}\label{radrank}
	Let $K$ be a field, and $R$ be a Cohen-Macaulay graded $K$-algebra. Then \[\max\{i \ | \ \df{i}{R}{K}(\omega_R)\neq 0 \} \leq \rra_K(R)-\dim(R).\] If $R=K[x_0,\dots,x_n]/I$ is standard graded, $X=\Proj(R)\subseteq \mathbb{P}^n$, and $\Sec(X)$ is the secant variety of $X$, then \[\max\{i \ | \ \df{i}{R}{K}(\omega_R)\neq 0 \} \leq \dim(\Sec(X)) - \dim(X).\]\qedhere
\end{proposition}
\begin{proof} Let $d=\dim(R)$.
	By \cite[Lemma~1.8.7.1]{EGAIV1}, $S\rightarrow R$ is radicial if and only if the map $R\otimes_S R\rightarrow R$ induces a surjection on spectra. We have isomorphisms \[R\otimes_S R \cong \ModDif{}{R}{K}/(\Delta_{S|K}\ModDif{}{R}{K})\] and $R \cong\ModDif{}{R}{K}/\Delta_{R|K}$, so $S\rightarrow R$ is radicial if and only if $\sqrt{\Delta_{S|K}\ModDif{}{R}{K}}=\sqrt{\Delta_{R|K}}$.
	
	Let $W$ be a canonical module for $\ModDif{}{R}{K}$. If  \[\df{i}{R}{K}(\omega_R)\cong H^{d+i}_{\Delta_{R|K}}(W)\neq 0 \quad \text{ and } \quad \sqrt{\Delta_{S|K}\ModDif{}{R}{K}}=\sqrt{\Delta_{R|K}},\] then $H^{d+i}_{\Delta_{S|K}}(W)\neq 0$. Consequently, $\Delta_{S|K}$ cannot be generated by fewer than $d+i$ elements. For any $n$-generated $K$-algebra, its diagonal ideal is generated by $n$ elements. Thus, $\df{i}{R}{K}(\omega_R)\neq 0$ implies $\rra_K(R)\geq d+i$, which establishes the first inequality of the statement.
\end{proof}

\begin{remark}
	We note that the local cohomology modules occurring in Proposition~\ref{functors-field} are closely related to those that appear in the work \cite{DufJef,DufJef2} of Emilie Dufresne and the present author on invariant theory. A \emph{separating set} for an action of a group $G$ on a polynomial ring $S=\Sym(V^*)$ induced by a representation $V$ of $G$ consists of a set of invariants $\{f_1,\dots, f_t\}\subset S^G$ such that for any $v,w\in V$ with $h(v)\neq h(w)$ for some $h\in S^G$, there is an $f_i$ with $f_i(v)\neq f_i(w)$; we refer the reader to \cite[Section~2.4]{DerKem} for an introduction to separating sets. By \cite[Theorem~2.2]{Duf}, if $G$ is reductive and $K=\bar{K}$, then $\{f_1,\dots, f_t\}$ is a separating set if and only if $K[f_1,\dots, f_t]\rightarrow S^G$ is radicial. In \cite{DufJef}, we show that the smallest cardinality of a separating set for a linear action of a finite group $G$ on a polynomial ring $R$ is bounded below by $\max\{ i \ | \ H^i_{\Delta_{R|K}}(\ModDif{}{S}{K})\neq 0\}$, where $R=S^G$. This is applied to give meaningful lower bounds in terms of properties of the representation. A similar approach is executed in \cite{DufJef2} to bound separating sets for actions of tori. It follows from Remark~\ref{CD} that the bounds given by Proposition~\ref{radrank} are always at least as strong as those mentioned above; they also do not require the ring $R$ to be realized as a subring of a polynomial ring.	\end{remark}
	
	Computation of cohomological dimension is difficult in general, so direct application of Proposition~\ref{radrank} to study radicial rank and secant dimension is challenging. Examples to this effect appear in the work \cite{DufJef,DufJef2} discussed above. We give a couple more applications based on the following corollary.

\begin{corollary}\label{cor-CM}
Let $R$ be a Gorenstein $\NN$-graded algebra of dimension $d>3$ with an isolated singularity, and $R_0=K$, a field. Suppose that $D_{R|K}$ has depth~$d$. Then, either the functor $D_{R|K}(-)$ is exact, or $\rra_K(R)\geq 2d-1$.
\end{corollary}
\begin{proof}
We have $H^i_{R_+}(D_{R|K})=0$ for all $i\leq d-1$, and hence by Corollary~\ref{depth-cor}, ${\df{i}{R}{K}(R)=0}$ for all $1 \leq i \leq d-2$. Thus, either $\df{i}{R}{K}(R)=0$ for all $i\neq 0$, in which case $D_{R|K}(-)$ is exact, or else $\max\{ i \ | \ \df{i}{R}{K}(R) \neq 0\} \in \{ d-1, d\}$, and the conclusion follows from Proposition~\ref{radrank}.
\end{proof}

\begin{example}
	Let $K$ be a field of characteristic zero, and $M$ be a $3\times 3$ matrix of indeterminates. Let $K[M]$ be the polynomial ring generated by the entries of $M$, and $R=K[M]/I_2(M)$, where $I_2(M)$ is the ideal of $2\times 2$ minors of $M$. The ring $R$ is a $5$-dimensional ring, with an isolated Gorenstein singularity at the homogeneous maximal ideal.  The secant variety of $X=\Proj(R)$ is $\Proj(K[M]/\det(M))$, which has dimension~$7$. By Proposition~\ref{radrank}, $\df{i}{R}{K}(R)=0$ for $i>3$.
	
	We claim that $\df{i}{R}{K}(R)\neq 0$ for some $i>0$. Equivalently, we show that ${H^i_{\Delta_{R|K}}(\ModDif{}{R}{K})\neq 0}$ for some $i>5$. Let $C_i\subset R$ for $i=1,2,3$ be the ideal generated by the images of the variables in the $i$th column of $M$. Consider the ideal
	\[ J = (C_2 \otimes 1 + C_3 \otimes 1 + 1 \otimes C_1 + 1 \otimes C_3) \subset \ModDif{}{R}{K}. \]
	The quotient ring $\ModDif{}{R}{K}/J$ is a polynomial ring in $6$ variables, and the image of $\Delta_{R|K}$ generates its homogeneous maximal ideal, so $H^6_{\Delta_{R|K}}(\ModDif{}{R}{K}/J)\neq 0$. Thus, by Remark~\ref{CD}, $H^i_{\Delta_{R|K}}(\ModDif{}{R}{K})\neq 0$ for some $i\geq 6$. Equivalently, $D_{R|K}(-)$ is not exact.
	
	Consequently, by Corollary~\ref{cor-CM}, $D_{R|K}$ is not Cohen-Macaulay. We note that $R$ is an invariant ring of an action of a torus $T$ on a polynomial ring $S$, and $D_{R|K}=D_{S|K}^T$ for this action \cite[Case A, Main Theorem
0.3, 0.7]{LS}. In particular, this gives a counterexample to the analogue of the Hochster--Roberts Theorem \cite{HocRob} for actions of linearly reductive groups on the Weyl algebra.

Note also that, by Remark~\ref{BZNrmk}, $R$ is an explicit example of a ring for which Theorem~\ref{functors-field} holds, but the Grothendieck-Sato theorem of \cite{BZN} does not apply.
\end{example}

\begin{example}\label{example-simplicial}
We can apply Corollary~\ref{cor-CM} to simplicial affine toric varieties with an isolated Gorenstein singularity. 

If $K$ is a field of positive characteristic $p$, by \cite[Theorem~7.2]{FMS}, $\End_{R^{p^e}}(R)$ is a Cohen-Macaulay $R$-module for each $e$. Then, by the isomorphism \eqref{eq:poschar} of the introduction, and passing to direct limits, one obtains that $H^i_{R_+}(D_{R|K})=0$ for all $i<d$. Thus, either $D_{R|K}(-)$ is exact, or $\rra_K(R) \geq 2d-1$.

To obtain a similar result in characteristic zero, we observe the following. If $R$ is a normal semigroup ring over a field $K$ of characteristic zero, there is a a semigroup ring $A$ over $\ZZ$ such that $A \otimes_{\ZZ} K\cong R$. If $R$ admits an injective map into affine $n$-space, the same holds for $A\otimes_{\ZZ} {\FF_p}$ for all but finitely many $p$; indeed, this is a consequence of the fact that injectivity of a given map can be characterized by existence of a solution to a finite set of equations.

Thus, for $R$ as above in characteristic zero, either the functor $D_{A/pA | \FF_p}(-)$ is exact for all but finitely many $p$, or $\rra_K(R)\geq 2d-1$.

We exhibit both cases explicitly. If $R$ is a polynomial ring, then $D_{R|K}(-)$ is exact. If $R$ is the $d$-th veronese ring of a polynomial ring of dimension $d$, and $d$ is invertible in $K$, then $R$ is a Gorenstein simplicial toric ring, and $\df{d-1}{R}{K}(R)\neq 0$ by \cite[Theorem~3.4]{DufJef} and Theorem~\ref{functors-field}, so $\rra_K(R)\geq 2d-1$.
\end{example}

The example above and Remark~\ref{BZNrmk} motivate the following question:

\begin{question}
For which $K$-algebras $R$ is the functor $D_{R|K}(-)$ exact?
\end{question}

We are not aware of any examples of finitely generated $K$-algebras for which $D_{R|K}(-)$ is exact that do not admit a radicial map to a polynomial ring.

%

\section{Reduction of differential operators to positive characteristic}

In this section, we apply Theorem~\ref{functors-field} to determine when differential operators behave well under reduction to positive characteristic. As mentioned in the introduction, this is governed by the first Smith--Van den Bergh functor:

\begin{proposition}[Smith--Van den Bergh {\cite[Subsection~5.1]{SmithVDB}}]\label{SVDB-reduction} Let $A$ be a Dedekind domain, and $R$ be an $A$-algebra such that $\ModDif{n}{R}{A}$ is a projective $A$-module for all $n$; in particular, $R$ is $A$-projective. Then, for every $A$-algebra $B$, there is a short exact sequence
	\[ 0 \longrightarrow \Diff{}{R}{A}(R) \otimes_A B \longrightarrow \Diff{}{(R\otimes_A B)\,}{\,B}(R\otimes_A B) \longrightarrow \Tor^A_{1}(B, \df{1}{R}{A}(R))\longrightarrow 0 . \]
\end{proposition}

\begin{proposition}\label{prop-reduction}
	Let $R$ be a Gorenstein $d$-dimensional Noetherian $\NN$-graded ring with $R_0=A$ a Dedekind domain. Suppose that $R$ is flat over $A$. Let $p\in A$ be a prime element. The following are equivalent:
	\begin{enumerate}
		\item The local cohomology module $H^d_{\Delta_{R|A}}(\ModDif{}{R}{A})$ has no $p$-torsion.
		\item There is an isomorphism $\Diff{}{\overline{R}}{\overline{A}}  \cong \Diff{}{R}{A}\otimes_{A} \overline{A}$, where $\overline{R}=R/pR$ and $\overline{A}=A/pA$. \qedhere
	\end{enumerate}
\end{proposition}
%
%
\begin{proof}
Set $A'=\widehat{A_{(p)}}$ and $R'=R\otimes_A A'$. Note that $A\rightarrow A'$ is faithfully flat. We have $\ModDif{}{R'}{A'}\cong \ModDif{}{R}{A}\otimes_A A'$ and $H^d_{\Delta_{R'|A'}}(\ModDif{}{R'}{A'})\cong H^d_{\Delta_{R|A}}(\ModDif{}{R}{A})\otimes_A A'$, using the behavior of local cohomology under flat base change. Consequently, existence of $p$-torsion in the modules $H^d_{\Delta_{R'|A'}}(\ModDif{}{R'}{A'})$ and $H^d_{\Delta_{R|A}}(\ModDif{}{R}{A})$ is equivalent.

By the flatness hypotheses, $p$ is a nonzerodivisor on $\ModDif{}{R'}{A'}$, so there is a short exact sequence
\[ 0 \xrightarrow{} \ModDif{}{R'}{A'} \xrightarrow{p}   \ModDif{}{R'}{A'}   \xrightarrow{} \ModDif{}{\overline{R}}{\overline{A}} \xrightarrow{} 0.\] There is then a long exact sequence of local cohomology
\begin{align*} \cdots \xrightarrow{} H^{d-1}_{\Delta_{R'|A'}}(\ModDif{}{R'}{A'}) &\xrightarrow{p} H^{d-1}_{\Delta_{R'|A'}}(\ModDif{}{R'}{A'}) \xrightarrow{} \\ &H^{d-1}_{\Delta_{\overline{R}|\overline{A}}}(\ModDif{}{\overline{R}}{\overline{A}}) \xrightarrow{} H^d_{\Delta_{R|A}}(\ModDif{}{R}{A})  \xrightarrow{p} H^d_{\Delta_{R|A}}(\ModDif{}{R}{A})  \xrightarrow{} \cdots,\end{align*}
which by Theorem~\ref{functors-field} yields
\[ \cdots \xrightarrow{}  D_{R' | A'} \xrightarrow{p} D_{R' | A'} \xrightarrow{}D_{\overline{R} | \overline{A}} \xrightarrow{} H^d_{\Delta_{R|A}}(\ModDif{}{R}{A})  \xrightarrow{p} H^d_{\Delta_{R|A}}(\ModDif{}{R}{A})  \xrightarrow{} \cdots.\]
Thus, the natural map $D_{R|A} \otimes_A \overline{A} \cong D_{R' | A'} \otimes_{A'} \overline{A}  \xrightarrow{}D_{\overline{R} | \overline{A}}$ is surjective if and only if multiplication by $p$ is injective on $H^d_{\Delta_{R|A}}(\ModDif{}{R}{A})$, as required.
\end{proof}

We can apply this observation to show that differential operators behave well under reduction to positive characteristic in certain circumstances.  The next three results provide positive answers to \cite[Question~5.1.2]{SmithVDB} under stronger hypotheses. In particular, the following theorem covers the case of \cite[Question~5.1.2]{SmithVDB} corresponding to the principal case of the main result of \cite{SmithVDB}; see \cite[Proposition~3.1.6]{SmithVDB}.

\begin{theorem}\label{splitting-reduction}
	Let $R$ be a graded subring of $S=\ZZ[x_1,\dots,x_m]$ such that $R$ is a direct summand of $S$ as an $R$-module. Suppose also that $R$ is Gorenstein. There is an $n\in A$ such that, after replacing $R$ by $R[1/n]$ and $\ZZ$ by $A=\ZZ[1/n]$, there are isomorphisms \[\Diff{}{\overline{R}}{\FF_p}  \cong \Diff{}{R}{A}\otimes_{A} \FF_p, \quad \text{where}\ \overline{R}=R/pR,\] for every prime $p\in \ZZ$ that does not divide $n$.
\end{theorem}
\begin{proof}
	Since $R$ is a graded direct summand of $S$, $\ModDif{}{R}{A}$ is a graded direct summand of $\ModDif{}{S}{A}$, as is $H^i_{\Delta_{R|A}}(\ModDif{}{R}{A})$ of $H^i_{\Delta_{R|A}}(\ModDif{}{S}{A})$ for each $i$. Since $\ModDif{}{S}{A}$ is a smooth $\ZZ$-algebra, by \cite{Bubbles}, $H^i_{\Delta_{R|A}}(\ModDif{}{S}{A})$ has $p$-torsion for only finitely many primes $p$ for each $i$. Thus, the same holds for $H^i_{\Delta_{R|A}}(\ModDif{}{R}{A})$. We can then apply Proposition~\ref{prop-reduction}.
\end{proof}

The main examples of direct summands of polynomial rings are rings of invariants of linearly reductive groups. One generally only discusses linearly reductive groups over a field, rather than over the integers. However, the same classes of rings that occur as invariants of linearly reductive groups in positive characteristic (extensions of tori by finite groups of invertible order) are covered by this theorem.

The author is not aware of a Gorenstein ring of positive characteristic that satisfies the assumptions of Smith and Van den Bergh's results on simplicity of rings of differential operators \cite[Theorems~1.3~and~4.2.1]{SmithVDB} that does not occur as a mod $p$ reduction of a ring satisfying the hypotheses of Theorem~\ref{splitting-reduction}.

\begin{corollary}
If $R$ is a Gorenstein normal semigroup ring over $\ZZ$, then there is an $n\in A$ such that, after replacing $R$ by $R[1/n]$ and $\ZZ$ by $A=\ZZ[1/n]$, there are isomorphisms \[\Diff{}{\overline{R}}{\FF_p}  \cong \Diff{}{R}{A}\otimes_{A} \FF_p, \quad \text{where}\ \overline{R}=R/pR,\] for every prime $p\in \ZZ$ that does not divide $n$.
\end{corollary}
\begin{proof}
In the terminology of \cite[\S2]{HochsterToric}, any normal semigroup can be emdedded as a full semigroup $S\subseteq\NN^d$ for some $d$ \cite[p.~323]{HochsterToric}. By \cite[Lemma~1]{HochsterToric}, for any $R$, the semigroup ring $R[S]$ generated by the monomials whose exponent vectors lie in $S$ is a direct summand as an $R[S]$-module of $R[\NN^d]=R[x_1,\dots,x_d]$. In particular, any normal semigroup ring over $\ZZ$ is a direct summand (as a module over itself) of a polynomial ring over $\ZZ$.
\end{proof}

For example, this corollary covers normal hypersurfaces given by a binomial equation, such as $\displaystyle \frac{\ZZ[u,v,w,x,y,z]}{(uvw-xyz)}$, as well as coordinate rings of Segre products $\mathbb{P}^n_{\ZZ} \times_{\ZZ} \mathbb{P}^n_{\ZZ}$.

\begin{corollary}\label{finitesummand}
Let $R$ be a graded subring of $S=\ZZ[x_1,\dots,x_m]$ such that $R$ is a direct summand of $S$ as an $R$-module.  Suppose that $R\otimes_{\ZZ} \QQ \rightarrow S\otimes_{\ZZ} \QQ$ is module-finite and splits as $R\otimes_{\ZZ} \QQ$-modules. Then there is an $n\in \ZZ$ such that, after replacing $R$ by $R[1/n]$ and $\ZZ$ by $A=\ZZ[1/n]$, there are isomorphisms \[\Diff{}{\overline{R}}{\FF_p}  \cong \Diff{}{R}{A}\otimes_{A} \FF_p, \quad \text{where}\ \overline{R}=R/pR,\] for every prime $p\in \ZZ$ that does not divide $n$. 
\end{corollary}
\begin{proof} Any fixed set of module generators $\{f_1,\dots,f_t\}$ of $S\otimes_{\ZZ} \QQ$ over $R\otimes_{\ZZ} \QQ$ live inside of $S[1/n]$ for some $n\in \ZZ$; after inverting an element of $\ZZ$, without loss of generality $\{f_1,\dots,f_t\}\subset S$. Applying generic freeness, we can localize further on $\ZZ$ so that $C=S/(\sum R f_i)$ is a free $A=\ZZ[1/n]$-module. Since $C\otimes_{\ZZ} \QQ=0$, we must have $C=0$. Thus, we can assume that $S$ is a finitely generated $R$-module.
	
	Now, let $f_1 \in S$ and $\phi:S\otimes_{\ZZ} \QQ \rightarrow R\otimes_{\ZZ} \QQ$ be such that $\phi(f_1\otimes 1)=1$. Note that the map $\phi$ is determined by the values $\phi(f_i \otimes 1)$. Since there are finitely many, we can again invert an element of $A$ to ensure each of these is contained in $R$. Then $\phi$ descends to a $R$-module splitting of the inclusion of $R$ into $S$. The result then follows from Theorem~\ref{splitting-reduction}.
\end{proof}

We can also apply Proposition~\ref{prop-reduction} to give a conceptual example of a local cohomology module with infinitely many associated primes; the first example of this phenomenon was given by Singh~\cite{Singh-LC}.

\begin{example}\label{assoc-primes}
	Let $\displaystyle R=\frac{\ZZ[x,y,z]}{(x^3+y^3+z^3)}$. In \cite{Smith-D}, Smith showed that for every $p\equiv 1 \, \mathrm{mod} \, 3$, there is a differential operator in $R \otimes_{\ZZ} \FF_p$ that is not the base change of a differential operator in $R$. The same argument with minor modifications works for all prime characteristics other than three:
	
	By \cite{BGG}, the differential operators on $R\otimes_{\ZZ} \CC$ of degree zero are generated (as an algebra) by the Euler operator, which multiplies any homogeneous element by its degree. Since $D_{R|\ZZ}$ is torsionfree over $\ZZ$, we have  $D_{R|\ZZ} \otimes_{\ZZ} \CC \cong D_{R\otimes \CC|\CC}$, so any element $\delta\in[D_{R|\ZZ}]_0$ is given by the rule $\delta(f)=P(\deg(f)) f$ for some polynomial $P\in \ZZ[x]$ for each homogeneous~$f$.
	
	Now, fix a prime $p\neq 3$, and set $\bar{R}=R\otimes \FF_p$. Both $\bar{R}$ and its ring of $p$th powers $\bar{R}^{p}$ are Gorenstein with $a$-invariant zero. Thus, there is a degree-preserving isomorphism ${\Hom_{\bar{R}^{p}}(\bar{R},\bar{R}^{p})\cong \bar{R}}$. In particular, there is a nonzero $\bar{R}^{p}$-linear map $\phi:\bar{R}\rightarrow \bar{R}$ of degree zero, the image of which is an ideal of $\bar{R}^{p}$. By the isomorphism \eqref{eq:poschar} of the introduction, we see that $\phi$ is a differential operator on $\bar{R}$ of degree zero. However, it follows from the previous paragraph that every degree zero element in the image of $D_{R|\ZZ} \otimes \FF_p$ must have as its own image a direct sum of graded pieces of $\bar{R}$. No ideal of $\bar{R}^p$ is of this form, so such a map $\phi$ cannot be the base change of an operator on $R$.
	
We identify $\ModDif{}{R}{\ZZ}$ with $\displaystyle \frac{\ZZ[x,y,z,x',y',z']}{(x^3+y^3+z^3,x'^3+y'^3+z'^3)}$, so that $x$ corresponds to $x\otimes 1$ and $x'$ corresponds to $1 \otimes x'$ in $R\otimes_{\ZZ} R$.
	It follows from Proposition~\ref{prop-reduction} that 
	\[\displaystyle H^3_{(x-x',y-y',z-z')}\left( \frac{\ZZ[x,y,z,x',y',z']}{(x^3+y^3+z^3,x'^3+y'^3+z'^3)}\right)\] has nonzero $p$-torsion for all primes~$p\neq 3$. In particular, this local cohomology module has infinitely many associated primes.
\end{example}

\section*{Acknowledgements}

The author is grateful to Alessandro De Stefani, Elo{\'i}sa Grifo, Mel Hochster, Jonathan Monta\~{n}o, Luis N{\'u}{\~n}ez-Betancourt, Anurag Singh, Karen Smith, Nick Switala, and Bernd Ulrich for many helpful and enjoyable discussions. We also thank the anonymous referees for helpful comments and suggestions that greatly improved the presentation. The author was partially supported by NSF Grant DMS~\#1606353.

\end{document}